\documentclass{cedram-afst-ppn}

\usepackage{setspace,xspace,graphicx}
\usepackage{amsmath,amsfonts,amssymb,amsxtra,setspace,xspace,color}
\usepackage[colorlinks=true]{hyperref}
\hypersetup{urlcolor=blue, citecolor=blue, linkcolor=blue}

\newtheorem{Theo}{Theorem}
\newtheorem{Lemm}[Theo]{Lemma}

\newtheorem{Coro}[Theo]{Corollary}
\newtheorem{Prop}[Theo]{Proposition}

\newcommand{\be}[1]{\begin{equation}\label{#1}}
\newcommand{\ee}{\end{equation}}
\renewcommand{\(}{\left(}
\renewcommand{\)}{\right)}
\newcommand{\dig}[1]{\int_{\partial\Omega}{#1}\;d{\mathcal H}^{d-1}}

\newcommand{\R}{{\mathbb R}}

\renewcommand{\H}{{\rm H}}
\renewcommand{\L}{{\rm L}}
\newcommand{\nrm}[2]{\|{#1}\|_{\L^{#2}(\Omega)}}
\newcommand{\intg}[1]{\int_{\Omega}#1\,dx}
\newcommand{\nrmrd}[2]{\|{#1}\|_{\L^{#2}(\R^d)}}

\title[Rigidity and optimal constants]{Uniqueness and rigidity in nonlinear elliptic equations, interpolation inequalities, and spectral estimates}
\alttitle{Unicit\'e et rigidit\'e pour des \'equations elliptiques non lin\'eaires, in\'egalit\'es d'interpolation et estimations spectrales}
\author{\firstname{Jean} \lastname{Dolbeault}}
\address{Ceremade, UMR CNRS n$^\circ$ 7534,\\ Universit\'e Paris-Dauphine, PSL research university,\\ Pl.~de Lattre de Tassigny, 75775 Paris C\'edex 16, France}
\email{dolbeaul@ceremade.dauphine.fr}
\author{\firstname{Micha\l} \lastname{Kowalczyk}}
\address{Departamento de Ingenier\'{\i}a Matem\'atica and Centro de Modela\-miento Matem\'atico (CNRS UMI n$^\circ$ 2807), Universidad de Chile, Casilla 170 Correo 3,\\ Santiago, Chile}
\email{kowalczy@dim.uchile.cl}

\thanks{J.D.~thanks the LabEx CIMI, the ANR projects STAB, NoNAP and Kibord for support, and the MathAmSud project QUESP. Both authors thank the ECOS (Chile-France) project C11E07. M.K.~thanks Chilean grants Fondecyt 1130126 and Fondo Basal CMM-Chile. Both authors are indebted to an anonymous referee for his careful reading of the paper and useful suggestions.}
\altkeywords{In\'egalit\'es fonctionnelles}

\begin{document}\vspace*{-1cm}
\begin{abstract} This paper is devoted to the Lin-Ni conjecture for a semi-linear elliptic equation with a super-linear, sub-critical nonlinearity and homogeneous Neumann boundary conditions. We establish a new rigidity result, that is, we prove that the unique positive solution is a constant if the parameter of the problem is below an explicit bound that we relate with an optimal constant for a Gagliardo-Nirenberg-Sobolev interpolation inequality and also with an optimal Keller-Lieb-Thirring inequality. Our results are valid in a sub-linear regime as well. The rigidity bound is obtained by nonlinear flow methods inspired by recent results on compact manifolds, which unify nonlinear elliptic techniques and the \emph{carr\'e du champ} method in semi-group theory. Our method requires the convexity of the domain. It relies on integral quantities, takes into account spectral estimates and provides improved functional inequalities.
\end{abstract}

%%%%%%%%%%%%%%%%%%%%%%%%%%%%%%%%%%%%%%%%%%%%%%%%%%%%%%%%%%%%%%%%%%%%%%%%%%%%%%
\begin{altabstract} Cet article est consacr\'e \`a la conjecture de Lin-Ni pour une \'equa\-tion semi-lin\'eaire elliptique avec non-lin\'earit\'e super-lin\'eaire, sous-critique et des conditions de Neumann homog\`enes. Nous \'etablissons un r\'esultat de rigidit\'e, c'est-\`a-dire nous prouvons que la seule solution positive est constante si le param\`etre du probl\`eme est en-dessous d'une borne explicite, reli\'ee \`a la constante optimale d'une in\'egalit\'e d'interpolation de Gagliardo-Ni\-renberg-Sobo\-lev et aussi \`a une in\'egalit\'e de Keller-Lieb-Thir\-ring optimale. Nos r\'esultats sont \'egalement valides dans un r\'egime sous-lin\'eaire. La borne de rigidit\'e est obtenue par des m\'ethodes de flots non-lin\'eaires inspir\'ees de r\'esultats r\'ecents sur les vari\'et\'es compactes, qui unifient des techniques d'\'equations elliptiques non-lin\'eaires et la m\'ethode du carr\'e du champ en th\'eorie des semi-groupes. Notre m\'ethode requiert la convexit\'e du domaine. Elle repose sur des quantit\'es int\'egrales, prend en compte des estimations spectrales et fournit des in\'egalit\'es am\'elior\'ees. \end{altabstract}

\maketitle
%%%%%%%%%%%%%%%%%%%%%%%%%%%%%%%%%%%%%%%%%%%%%%%%%%%%%%%%%%%%%%%%%%%%%%%%%%%%%%

%%%%%%%%%%%%%%%%%%%%%%%%%%%%%%%%%%%%%%%%%%%%%%%%%%%%%%%%%%%%%%%%%%%%%%%%%%%%%%
%%%%%%%%%%%%%%%%%%%%%%%%%%%%%%%%%%%%%%%%%%%%%%%%%%%%%%%%%%%%%%%%%%%%%%%%%%%%%%
\section{Introduction and main results}\label{Sec:Intro}

Let us assume that $\Omega$ is a bounded domain in $\R^d$ with smooth boundary. To avoid normalization issues, we shall assume throughout this paper that
\[
|\Omega|=1\,.
\]
The unit outgoing normal vector at the boundary is denoted by $n$ and $\partial_nu=\nabla u\cdot n$. We shall denote by $2^*=\frac{2\,d}{d-2}$ the critical exponent if $d\ge3$ and let $2^*=\infty$ if $d=1$ or $2$. Assume first that $p$ is in the range $1<p<2^*-1=(d+2)/(d-2)$ if $d\ge3$, $1<p<\infty$ if $d=1$ or $2$, and let us consider the three following problems.
\begin{enumerate}
\item[(P1)] For which values of $\lambda>0$ does the equation
\be{Eq1}
-\,\Delta u+\lambda\,u=u^p\quad\mbox{in}\quad\Omega\,,\quad\partial_nu=0\quad\mbox{on}\quad\partial\Omega
\ee
has a unique positive solution~?
\item[(P2)] For any $\lambda>0$, let us define
\[
\mu(\lambda):=\inf_{u\in\H^1(\Omega)\setminus\{0\}}\frac{\nrm{\nabla u}2^2+\lambda\,\nrm u2^2}{\nrm u{p+1}^2}\,.
\]
For which values of $\lambda>0$ do we have $\mu(\lambda)=\lambda$~?
\item[(P3)] Assume that $\phi$ is nonnegative function in $\L^q(\Omega)$ with $q=\frac{p+1}{p-1}$ and denote by $\lambda_1(\Omega,-\phi)$ the lowest eigenvalue of the Schr\"odinger operator $-\Delta-\phi$. Let us consider the optimal inequality
\[
\lambda_1(\Omega,-\phi)\ge-\,\nu\big(\nrm\phi q\big)\quad\forall\,\phi\in\L_+^q(\Omega)
\]
For which values of $\mu$ do we know that $\nu(\mu)=\mu$~? 
\end{enumerate}
The three problems are related. Uniqueness in (P1) means that $u=\lambda^{1/(p-1)}$ while equality cases $\mu(\lambda)=\lambda$ in (P2) and $\nu(\mu)=\mu$ and (P3) are achieved by constant functions and constant potentials respectively. We define a threshold value $\mu_i$ with $i=1$, $2$, $3$ such that the answer to (P$i$) is yes if $\mu<\mu_i$ and no if $\mu>\mu_i$.

\medskip Our method is not limited to the case $p>1$. If $p$ is in the range $0<p<1$, the three problems can be reformulated as follows.
\begin{enumerate}
\item[(P1)] For which values of $\lambda>0$ does the equation
\be{Eq2}
-\,\Delta u+u^p=\lambda\,u\quad\mbox{in}\quad\Omega\,,\quad\partial_nu=0\quad\mbox{on}\quad\partial\Omega
\ee
has a unique nonnegative solution~?
\item[(P2)] For any $\mu>0$, let us define
\[
\lambda(\mu):=\inf_{u\in\H^1(\Omega)\setminus\{0\}}\frac{\nrm{\nabla u}2^2+\mu\,\nrm u{p+1}^2}{\nrm u2^2}\,.
\]
For which values of $\mu>0$ do we have $\lambda(\mu)=\mu$~?
\item[(P3)] Assume that $\phi$ is nonnegative function in $\L^q(\Omega)$ with $q=\frac{1+p}{1-p}$ and still denote by $\lambda_1(\Omega,\phi)$ the lowest eigenvalue of the Schr\"odinger operator $-\Delta+\phi$. Let us consider the optimal inequality
\[
\lambda_1(\Omega,\phi)\ge\nu(\nrm{\phi^{-1}}q)\quad\forall\,\phi\in\L_+^q(\Omega)
\]
For which values of $\mu$ do we know that $\nu(\mu)=\mu$~? 
\end{enumerate}

\medskip The problems of the range $0<p<1$ and $1<p<2^*$ can be unified. Let us define
\[
\varepsilon(p)=\frac{p-1}{|p-1|}
\]
and observe that $\nrm u{p+1}\le\nrm u2$ if $p<1$, $\nrm u2\le\nrm u{p+1}$ if $p>1$, so that $\varepsilon(p)\,\big(\nrm u{p+1}-\nrm u2\big)$ is nonnegative. Our three problems can be reformulated as follows.
\begin{enumerate}
\item[(P1)] Let us consider the equation
\be{Eq3}
-\,\varepsilon(p)\,\Delta u+\lambda\,u-u^p=0\quad\mbox{in}\quad\Omega\,,\quad\partial_nu=0\quad\mbox{on}\quad\partial\Omega
\ee
and define
\[
\mu_1:=\inf\{\lambda>0\,:\,\mbox{\eqref{Eq3} has a non constant positive solution}\}\,.
\]
We shall say that \emph{rigidity} holds in~\eqref{Eq3} if $u=\lambda^{1/(p-1)}$ is its unique positive solution. 
\item[(P2)] For any $\mu>0$, take $\lambda(\mu)$ as the best (i.e. the smallest if $\varepsilon(p)>0$ and the largest if $\varepsilon(p)<0$) constant in the inequality
\be{P2}
\nrm{\nabla u}2^2\ge\varepsilon(p)\,\Big[\mu\,\nrm u{p+1}^2-\lambda(\mu)\,\nrm u2^2\Big]\quad\forall\,u\in\H^1(\Omega)\,.
\ee
Here we denote by $\mu\mapsto\lambda(\mu)$ the inverse function of $\lambda\mapsto\mu(\lambda)$. Let
\[
\mu_2:=\inf\{\lambda>0\,:\,\mu(\lambda)\neq\lambda\mbox{ in~\eqref{P2}}\}\,.
\]
\item[(P3)] Let us consider the optimal \emph{Keller-Lieb-Thirring inequality}
\be{P3}
\nu(\mu)=-\,\varepsilon(p)\,\inf_{\phi\in\mathcal A_\mu}\lambda_1(\Omega,-\,\varepsilon(p)\,\phi)
\ee
where the admissible set for the potential $\phi$ is defined by
\[
\mathcal A_\mu:=\left\{\phi\in\L_+^q(\Omega)\,:\,\nrm{\phi^{\varepsilon(p)}}q=\mu\right\}
\]
and $q=(p+1)/|p-1|$. Let
\[
\mu_3:=\inf\{\mu>0\,:\,\nu(\mu)\neq\mu\mbox{ in~\eqref{P3}}\}\,.
\]
\end{enumerate}
Finally let us define $\Lambda_\star$ as the \emph{best constant in the interpolation inequality}
\be{GenInterpIneq}
\nrm{\nabla u}2^2\ge\frac{\Lambda_\star}{p-1}\,\Big[\nrm u{p+1}^2-\nrm u2^2\Big]\quad\forall\,u\in\H^1(\Omega)\,.
\ee
Let us observe that $\Lambda_\star$ may depend on $p$.
%-----------------------------------------------------------------------------
\begin{Theo}\label{Thm:Main1} Assume that $d\ge2$, $p\in(0,1)\cup(1,2^*-1)$ and $\Omega$ is a bounded domain in $\R^d$ with smooth boundary such that $|\Omega|=1$. With the above notations, we have
\[\label{Ineq:Main1}
0<\mu_1\le\mu_2=\mu_3=\frac{\Lambda_\star}{|p-1|}
\]
and, with $\lambda(\mu)$ and $\nu(\mu)$ defined as in~\eqref{P2} and~\eqref{P3}, the following properties hold:
\begin{enumerate}
\item[(P1)] Rigidity holds in~\eqref{Eq3} for any $\lambda\in(0,\mu_1)$.
\item[(P2)] The function $\mu\mapsto\lambda(\mu)$ is monotone increasing, concave if $p\in(0,1)$, convex if $p\in(1,2^*-1)$ and $\lambda(\mu)=\mu$ if and only if $\mu\le\mu_2$.
\item[(P3)] For any $\mu>0$, $\nu(\mu)=\lambda(\mu)$.
\end{enumerate}
\end{Theo}
%-----------------------------------------------------------------------------
This result is inspired from a series of recent papers on interpolation inequalities, rigidity results and Keller-Lieb-Thirring estimates on compact manifolds. Concerning Keller-Lieb-Thirring inequalities, we refer to~\cite{DolEsLa-APDE2014,Dolbeault2013437,MR3377678}, and to the initial paper~\cite{MR0121101} by J.B.~Keller whose results were later rediscovered by E.H.~Lieb and W.~Thirring in~\cite{Lieb-Thirring76}. For interpolation inequalities on compact manifolds, we refer to \cite{Dolbeault20141338,DEKL,dolbeault:hal-01206975} and references therein. In our case, the absence of curvature and the presence of a boundary induce a number of changes compared to these papers, that we shall study next. Beyond the properties of Theorem~\ref{Thm:Main1} which are not very difficult to prove, our main goal is to get explicit estimates of $\mu_i$ and $\Lambda_\star$.

Let us define
\[
\lambda_2:=\lambda_2(\Omega,0)
\]
which is the second (and first positive) eigenvalue of $-\,\Delta$ on $\Omega$, with homogeneous Neumann boundary conditions. Recall that the lowest eigenvalue of $-\,\Delta$ is $\lambda_1=0$ and that the corresponding eigenspace is spanned by the constants. For this reason $\lambda_2$ is often called the \emph{spectral gap} and the Poincar\'e inequality sometimes appears in the literature as the \emph{spectral gap inequality}. Finally let us introduce the number
\be{thetastar}
\theta_\star(p,d)=\frac{(d-1)^2\,p}{d\,(d+2)+p}\,.
\ee
%-----------------------------------------------------------------------------
\begin{Theo}\label{Thm:Main2} Assume that $d\ge1$ and $\Omega$ is a bounded domain in $\R^d$ with smooth boundary such that $|\Omega|=1$. With the above notations, we get the following estimates:
\[
\frac{1-\theta_\star(p,d)}{|p-1|}\,\lambda_2\le\mu_1\le\mu_2=\mu_3=\frac{\Lambda_\star}{|p-1|}\le\frac{\lambda_2}{|p-1|}
\]
for any $p\in(0,1)\cup(1,2^*\!-\!1)$. The lower estimate holds only under the additional assumptions that $\Omega$ is convex and $d\ge2$.\end{Theo}
%-----------------------------------------------------------------------------
Before giving a brief overview of the literature related to our results, let us emphasize two points. We first notice that $\lim_{p\to(d+2)/(d-2)}\theta_\star(p,d)=1$ if $d\ge3$ so that the lower estimate goes to $0$ as the exponent $p$ approaches the critical exponent. This is consistent with the previously known results on rigidity, that are based on Morrey's scheme and deteriorate as $p$ approaches $(d+2)/(d-2)$. In the critical case, multiplicity may hold for any value of~$\lambda$, so that one cannot expect that rigidity could hold without an additional assumption. The second remark is the fact that the convexity of~$\Omega$ is essential for known results in the critical case and one should not be surprised to see this condition also in the sub-critical range. This assumption is however not required in the result of C.-S.~Lin, W.-M.~Ni, and I.~Takagi in~\cite{MR929196}. Compared to their paper, what we gain here when $p>1$ is a fully explicit estimate which relies on a simple computation. The case $p<1$ has apparently not been studied yet.

It is remarkable that the case $p=1$ is the endpoint of the two admissible intervals in $p$. We may notice that
\[
\theta_\star(1,d)=\frac{(d-1)^2}{(d+1)^2}
\]
is in the interval $(0,1)$ for any $d\ge2$. The case $p=1$ is a limit case, which corresponds to the logarithmic Sobolev inequality
\be{LogSob}
\nrm{\nabla u}2^2-\frac{\Lambda_\star}2\intg{|u|^2\,\log\(\frac{|u|^2}{\nrm u2^2}\)}\ge0\quad\forall\,u\in\H^1(\Omega)
\ee
where $\Lambda_\star$ denotes the optimal constant, and by passing to the limit as $p\to1$ in~\eqref{GenInterpIneq}, we have the following result.
%-----------------------------------------------------------------------------
\begin{Coro}\label{Cor:Main3} If $d\ge2$ and $\Omega$ is a bounded domain in $\R^d$ with smooth boundary such that $|\Omega|=1$, then
\[
\frac{4\,d}{(d+1)^2}\,\lambda_2\le\Lambda_\star\le\lambda_2\,.
\]\end{Coro}
%-----------------------------------------------------------------------------
It is also possible to define a family of logarithmic Sobolev inequalities depending on $\lambda$, or to get a parametrized Keller-Lieb-Thirring inequality and find that \hbox{$\lambda=\Lambda_\star$} corresponds to a threshold value between a linear dependence of the optimal constant in $\lambda$ and a regime in which this dependence is given by a strictly convex function of $\lambda$. The interested reader is invited to refer to~\cite[Corollaries 13 and 14]{DolEsLa-APDE2014} for similar results on the sphere.

\medskip The existence of a non-trivial solution to~\eqref{Eq1} bifurcating from the constant ones for $\lambda=\lambda_2/(p-1)$ has been established for instance in~\cite{MR974610} when $p>1$. This paper also contains the conjecture, known in the literature as the \emph{Lin-Ni conjecture} and formulated in~\cite{MR974610}, that there are no nontrivial solutions for $\lambda>0$ small enough and that there are non-trivial solutions for $\lambda$ large enough, even in the super-critical case $p>2^*-1$. More details can be found in~\cite{MR929196}. Partial results were obtained before in~\cite{MR849484}, when the exponent is in the range $1<p<d/(d-2)$. These papers were originally motivated by the connection with the model of Keller and Segel in chemotaxis and the Gierer-Meinhardt system in pattern formation: see~\cite{MR1490535} for more explanations.

For completeness, let us briefly review what is known in the critical case $p=2^*-1$. When $d=3$, it was proved by M.~Zhu in~\cite{MR1691074} that rigidity holds true for $\lambda>0$ small enough when one considers the positive solutions to the nonlinear elliptic equation
\[
\Delta u-\lambda\,u+f(u)=0
\]
on a smooth bounded domain of $\R^3$ with homogeneous Neumann boundary conditions, if $f(u)$ is equal to $u^5$ up to a perturbation of lower order. Another proof was given by J.~Wei and X.~Xu in~\cite{MR2194150} and slightly extended later in~\cite{MR2410880}. The Lin-Ni conjecture is wrong for $p=2^*-1$ in higher dimensions: see~\cite{MR2806101}, and also earlier references therein. Some of the results have been extended to the $d$-Laplacian in dimension $d$ in~\cite{MR2446347}.

Compared to the case of homogeneous Dirichlet boundary conditions or in the whole space, much less is known concerning bifurcations, qualitative aspects of the branches of solutions and multiplicity in the case of homogeneous Neumann boundary conditions. We may for instance refer to~\cite{MR607786,MR1090827,MR2484935,MR1897394,Miyamoto20101853} for some results in this direction, but only in rather simple cases (balls, intervals or rectangles).

Concerning the Lin-Ni conjecture, it is known from~\cite[Theorem~3,~(ii)]{MR929196} that $u\equiv\lambda^{1/(p-1)}$ if $\lambda$ is small enough (also see~\cite{MR849484} for an earlier partial result), and that there is a non-trivial solution if $\lambda$ is large enough. As already said above, the method is based on the Moser iteration technique, in order to get a uniform estimate on the solution, and then on a direct estimate based on the Poincar\'e inequality. In the proof of Theorem~\ref{Thm:Main2} we shall adopt a completely different strategy, which is inspired by the \emph{rigidity} results for nonlinear elliptic PDEs as in~\cite{MR615628,MR1134481} on the one hand, and by the \emph{carr\'e du champ} method of D.~Bakry and M.~Emery on the other hand, that can be traced back to~\cite{MR772092}. More precisely, we shall rely on improved versions of these methods as in~\cite{MR1412446,MR1631581}, which involve the eigenvalues of the Laplacian, results on interpolation inequalities on compact manifolds obtained by J.~Demange in~\cite{MR2381156}, and a recent improvement with a computation based on traceless Hessians in~\cite{Dolbeault20141338}.

{F}rom a larger perspective, our approach in Theorem~\ref{Thm:Main2} is based on results for compact Riemannian manifolds that can be found in various papers: the most important ones are the rigidity results of L.~V\'eron et al.~in~\cite{MR1134481,MR1338283}, the computations inspired by the \emph{carr\'e du champ} method of~\cite{MR1412446,MR2381156}, and the nonlinear flow approach of~\cite{Dolbeault20141338} (also see~\cite{DEKL,DEKL2012,Dolbeault06082014}). Using these estimates in the range $1<p<2^*-1$ and the Bakry-Emery method as in~\cite{pre05312043} in the case $p\in(0,1)$, our goal is to prove that rigidity holds in a certain range of~$\lambda$ without relying on uniform estimates (and the Moser iteration technique) and discuss the estimates of the threshold values. The spectral estimates of Theorem~\ref{Thm:Main1} are directly inspired by~\cite{DolEsLa-APDE2014,Dolbeault2013437}.

\medskip This paper is organized as follows. Preliminary results have been collected in Section~\ref{Sec:Prelim}. The proof of Theorem~\ref{Thm:Main1} is given in Section~\ref{Sec:Equiv}. In Section~\ref{Sec:Linear}, we use the heat flow to establish a first lower bound similar to the one of Theorem~\ref{Thm:Main2}. Using a nonlinear flow a better bound is obtained in Section~\ref{Sec:Nonlinear}, which completes the proof of Theorem~\ref{Thm:Main2}. The last section is devoted to various considerations on flows and, in particular, to improvements based on the nonlinear flow method.

%%%%%%%%%%%%%%%%%%%%%%%%%%%%%%%%%%%%%%%%%%%%%%%%%%%%%%%%%%%%%%%%%%%%%%%%%%%%%%
\subsection*{Notations}

If $A=(A_{ij})_{1\le i,j\le d}$ and $B=(B_{ij})_{1\le i,j\le d}$ are two matrices, let $A\!:\!B=\sum_{i,j=1}^dA_{ij}\,B_{ij}$ and $|A|^2=A\!:\!A$. If $a$ and $b$ take values in $\R^d$, we adopt the definitions:
\begin{multline*}
a\cdot b=\sum_{i=1}^da_i\,b_i\;,\quad\nabla\cdot a=\sum_{i=1}^d\frac{\partial a_i} {\partial x_i}\;,\\
 a\otimes b=(a_i\,b_j)_{1\le i,j\le d}\;,\quad\nabla\otimes a=\(\frac{\partial a_j}{\partial x_i}\)_{1\le i,j\le d}\,.
\end{multline*}

%%%%%%%%%%%%%%%%%%%%%%%%%%%%%%%%%%%%%%%%%%%%%%%%%%%%%%%%%%%%%%%%%%%%%%%%%%%%%%
%%%%%%%%%%%%%%%%%%%%%%%%%%%%%%%%%%%%%%%%%%%%%%%%%%%%%%%%%%%%%%%%%%%%%%%%%%%%%%
\section{Preliminary results}\label{Sec:Prelim}

Let us recall that $\Omega\subset\R^d$, $d\ge 2$ is a bounded domain with smooth boundary (or an open interval if $d=1$) and let $\lambda_2$ be the first non-zero eigenvalue of the Laplace operator on $\Omega$, supplemented with homogeneous Neumann boundary conditions. We shall denote by $n$ a unit outgoing normal vector of $\partial\Omega$ and will denote by $u_2\in\H^1(\Omega)$ a non-trivial eigenfunction associated with the lowest positive eigenvalue $\lambda_2$, so that
\be{Eqn:FirstEigenvalue}
-\Delta u_2=\lambda_2\,u_2\quad\mbox{in}\quad\Omega\;,\quad\partial_nu_2=0\quad\mbox{on}\quad\partial\Omega\,.
\ee
As a trivial observation we may observe that $u_2$ is in $\H^2(\Omega)$.
%-----------------------------------------------------------------------------
\begin{Lemm}\label{Lem:LinInterp} With the above notations, for any $u\in\H^2(\Omega)$ such that $\partial_nu=0$ on $\partial\Omega$, we have
\[
\(\intg{|\nabla u|^2}\)^2\le\intg{|\Delta u|^2}\intg{|u|^2}\,.
\]
As a consequence, we also have
\be{Ineq:First}
\lambda_2\intg{|\nabla u|^2}\le\intg{|\Delta u|^2}\,,
\ee
and equality holds for any eigenfunction associated with $\lambda_2$. 
\end{Lemm}
%-----------------------------------------------------------------------------
\begin{proof} By expanding the square and integrating by parts the cross term, we notice that
\begin{multline*}
0\le\intg{\left|\frac1{\sqrt\mu}\,\Delta u+\lambda_2\,\sqrt\mu\,u\right|^2}\\
=\frac1\mu\,\nrm{\Delta u}2^2-2\,\lambda_2\,\nrm{\nabla u}2^2+\lambda_2^2\,\mu\,\nrm u2^2\,,
\end{multline*}
where $\mu$ is an arbitrary positive real parameter. After optimizing on $\mu>0$, we arrive at
\[
0\le2\,\lambda_2\(\nrm{\Delta u}2\,\nrm u2-\nrm{\nabla u}2^2\)\,.
\]
To check the equality case with $u=u_2$, it is enough to multiply~\eqref{Eqn:FirstEigenvalue} by $u_2$ and by $-\Delta u_2$, and then integrate by parts. By definition of $\lambda_2$, we know that
\[
\intg{|\nabla u|^2}\ge\lambda_2\intg{|u|^2}\quad\mbox{if}\quad\intg{u}=0
\]
with equality again if $u=u_2$. This concludes the proof. Notice indeed that the condition $\intg{u}=0$ can always be imposed without loss of generality, by adding the appropriate constant to $u$.\end{proof}

\medskip A key result for this paper is based on the computation of $(\Delta u)^2$ in terms of the Hessian matrix of $u$, which involves integrations by parts and boundary terms. The following result can be found in~\cite[Lemma 5.1]{GTS} or~\cite{MR775683}.
%-----------------------------------------------------------------------------
\begin{Lemm}\label{Lem:Convex} If $\Omega$ is a smooth convex domain in $\R^d$ and if $u\in C^3(\overline\Omega)$ is such that $\partial_nu=0$ on $\partial\Omega$, then
\[
-\sum_{i,j=1}^d\dig{\partial_{ij}^2u\,\partial_i u\, n_j}\geq 0\,.
\]
As a consequence, if $u\in\mathrm H^2(\Omega)$ is such that $\partial_nu=0$, then we have that
\[
\intg{|\Delta u|^2}\ge\intg{|\mathrm{Hess}\,u|^2}\,.
\]
\end{Lemm}
%-----------------------------------------------------------------------------
In Lemma~\ref{Lem:Convex}, the convexity is an essential ingredient, and this is where the convexity assumption comes from in all results of this paper.

Consider on $\H^1(\Omega)$ the functional
\be{JLambda}
\mathcal J_\Lambda[u]:=\nrm{\nabla u}2^2-\frac\Lambda{p-1}\,\left[\nrm u{p+1}^2-\nrm u2^2\right]
\ee
if $p\neq1$, and
\[
\mathcal J_\Lambda[u]:=\nrm{\nabla u}2^2-\frac\Lambda2\intg{|u|^2\,\log\(\frac{|u|^2}{\nrm u2^2}\)}
\]
if $p=1$.
%-----------------------------------------------------------------------------
\begin{Lemm}\label{Lem:Upper} There exists a function $u\in\H^1(\Omega)$ such that $\partial_nu=0$ and $\mathcal J_\Lambda[u]<0$ if $\Lambda>\lambda_2$.\end{Lemm}
%-----------------------------------------------------------------------------
\begin{proof} A simple computations shows that
\[
\mathcal J_\Lambda[1+\epsilon\,w]\sim\epsilon^2\,\left[\nrm{\nabla w}2^2-\Lambda\,\nrm w2^2\right]
\]
as $\epsilon\to0$. By choosing $w=u_2$ to be an eigenfunction associated with $\lambda_2$, we get that 
\[
\mathcal J_\Lambda[1+\epsilon\,w]\sim\epsilon^2\(\lambda_2-\Lambda\)\,\nrm w2^2
\]
is negative for $\epsilon>0$ small enough if $\Lambda>\lambda_2$.
\end{proof}
Lemma~\ref{Lem:Upper} provides the upper bound in Theorem~\ref{Thm:Main2}. Indeed this proves that
\[
\Lambda_\star\le\lambda_2\,.
\]
This method has been widely exploited and a similar argument can be found for instance in~\cite{MR849484}.

%%%%%%%%%%%%%%%%%%%%%%%%%%%%%%%%%%%%%%%%%%%%%%%%%%%%%%%%%%%%%%%%%%%%%%%%%%%%%%
%%%%%%%%%%%%%%%%%%%%%%%%%%%%%%%%%%%%%%%%%%%%%%%%%%%%%%%%%%%%%%%%%%%%%%%%%%%%%%
\section{Proof of Theorem~\ref{Thm:Main1}}\label{Sec:Equiv}

Assume that $p>1$ and let us recall that
\[
\mu(\lambda):=\inf_{u\in\H^1(\Omega)\setminus\{0\}}\mathcal Q_\lambda[u]\quad\mbox{with}\quad\mathcal Q_\lambda[u]:=\frac{\nrm{\nabla u}2^2+\lambda\,\nrm u2^2}{\nrm u{p+1}^2}\,.
\]
We denote by $\kappa_{p,d}$ the optimal constant in the following Gagliardo-Niren\-berg inequality on $\R^d$:
\[
\nrmrd{\nabla v}2^2+\nrmrd v2^2\ge\kappa_{p,d}\,\nrmrd v{p+1}^2\quad\forall\,v\in\H^1(\R^d)\,.
\]
%-----------------------------------------------------------------------------
\begin{Lemm}\label{Lem:1} If $p\in(1,2^*-1)$, the function $\lambda\mapsto\mu(\lambda)$ is monotone increasing, concave, such that $\mu(\lambda)\le\lambda$ for any $\lambda>0$, and $\mu(\lambda)=\lambda$ if and only if $0<\lambda\le\mu_2=\Lambda_\star/(p-1)$. Moreover, we have
\[
\mu(\lambda)\sim2^\frac{1-p}{1+p}\,\kappa_{p,d}\,\lambda^{1-\frac d2\,\frac{p-1}{p+1}}\quad\mbox{as}\quad\lambda\to+\infty\,.
\]
\end{Lemm}
%-----------------------------------------------------------------------------
\begin{proof} For any given $u\in\H^1(\Omega)\setminus\{0\}$, $\lambda\mapsto \mathcal Q_\lambda[u]$ is affine, increasing. By taking an infimum, we know that, as a function of $\lambda$, $\mu$ is concave, non-decreasing. Using $u\equiv1$ as a test function, we know that $\mu(\lambda)\le\lambda$ for any $\lambda>0$. By standard variational methods, we know that there is an optimal function $u\in\H^1(\Omega)\setminus\{0\}$, so that
\[
\nrm{\nabla u}2^2+\lambda\,\nrm u2^2=\mu(\lambda)\,\nrm u{p+1}^2\,.
\]
On the other hand, we know from~\eqref{GenInterpIneq} that
\[
\nrm{\nabla u}2^2+\frac{\Lambda_\star}{p-1}\,\nrm u2^2\ge\frac{\Lambda_\star}{p-1}\,\nrm u{p+1}^2\,.
\]
Hence we have the inequality
\[
\(1-\frac\lambda{\Lambda_\star}\,(p-1)\)\nrm{\nabla u}2^2\le\(\mu(\lambda)-\lambda\)\,\nrm u{p+1}^2\,.
\]
If $\lambda\le\Lambda_\star/(p-1)$, the l.h.s.~is nonnegative while the r.h.s.~is nonpositive because $\mu(\lambda)\le\lambda$, so that we conclude at once that $\mu(\lambda)=\lambda$ and $u$ is constant. As a consequence, $\mu_2\ge\Lambda_\star/(p-1)$. On the other hand, by definition of $\Lambda_\star$, we know that $\mu_2\le\Lambda_\star/(p-1)$.

The regime as $\lambda\to\infty$ is easily studied by a rescaling. If $u_\lambda$ denotes an optimal function such that $\mathcal Q_\lambda[u_\lambda]=\mu(\lambda)$, then $v_\lambda(x):=u_\lambda(x/\sqrt\lambda)$ is an optimal function for
\[
\|\nabla v\|_{\L^2(\Omega_\lambda)}^2+\|v\|_{\L^2(\Omega_\lambda)}^2\ge\frac{\mu(\lambda)}{\lambda^{1-\frac d2\,\frac{p-1}{p+1}}}\,\|v\|_{\L^{p+1}(\Omega_\lambda)}^2\quad\forall\,v\in\H^1(\Omega_\lambda)
\]
where $\Omega_\lambda:=\{x\in\R^d\,:\,\lambda^{-1/2}\,x\in\Omega\}$. Using truncations of the optimal functions for the Gagliardo-Nirenberg inequality on $\R^d_+=\{(x_1,x_2,...x_d)\in\R^d\,:\,x_d>0\}$ and an analysis of the convergence of an extension of $v_\lambda$ in $\H^1(\R^d)$ as $\lambda\to\infty$ based on standard concentration-compactness methods, up to the extraction of subsequences and translations, we get that the limit function $v$ is optimal for the inequality
\[
\|\nabla v\|_{\L^2(\R^d_+)}^2+\|v\|_{\L^2(\R^d_+)}^2\ge2^\frac{1-p}{1+p}\,\kappa_{p,d}\,\|v\|_{\L^{p+1}(\R^d_+)}^2\quad\forall\,v\in \H^1(\R^d_+)\,.
\]
See~\cite[Lemma~5]{DolEsLa-APDE2014} for more details in a similar case.

By definition, $\lambda\mapsto\mu(\lambda)$ is monotone non-decreasing. As a consequence of the behavior at infinity and of the concavity property, this monotonicity is strict. Hence $\mu$ is a monotone increasing function of $\lambda$.
\end{proof}

Assume that $p<1$ and let us recall that
\[
\lambda(\mu):=\inf_{u\in\H^1(\Omega)\setminus\{0\}}\mathcal Q^\mu[u]\quad\mbox{with}\quad\mathcal Q^\mu[u]:=\frac{\nrm{\nabla u}2^2+\mu\,\nrm u{p+1}^2}{\nrm u2^2}\,.
\]
We denote by $\kappa_{p,d}^+$ the optimal constant in the following Gagliardo-Niren\-berg inequality on $\R^d_+$:
\[
\|\nabla v\|_{\L^2(\R^d_+)}^2+\|v\|_{\L^{p+1}(\R^d_+)}^2\ge\kappa_{p,d}^+\,\|v\|_{\L^2(\R^d_+)}^2\quad\forall\,v\in\H^1(\R^d_+)\,.
\]
%-----------------------------------------------------------------------------
\begin{Lemm}\label{Lem:2} If $p\in(0,1)$, the function $\mu\mapsto\lambda(\mu)$ is monotone increasing, concave, such that $\lambda(\mu)\le\mu$ for any $\mu>0$, and $\lambda(\mu)=\mu$ if and only if $0<\mu\le\mu_2=\Lambda_\star/(1-p)$. Moreover, we have
\[
\lambda(\mu)\sim\kappa_{p,d}^+\,\mu^{\big(1+\frac d2\,\frac{1-p}{p+1}\big)^{-1}}\quad\mbox{as}\quad\mu\to+\infty\,.
\]
\end{Lemm}
%-----------------------------------------------------------------------------
\begin{proof} The proof follows the same strategy as in the proof of Lemma~\ref{Lem:1}. See~\cite[Lemma~11]{DolEsLa-APDE2014} for more details in a similar case. \end{proof}

Recall that we denote by $\mu\mapsto\lambda(\mu)$ the inverse function of $\lambda\mapsto\mu(\lambda)$ and get in both cases, $p>1$ and $p<1$, the fact that
\[
\mu(\lambda)=O\(\lambda^{1-\frac d2\,\frac{p-1}{p+1}}\)\quad\mbox{as}\quad\lambda\to+\infty\,.
\]
%-----------------------------------------------------------------------------
\begin{Lemm}\label{Lem:3} Under the assumptions of Theorem~\ref{Thm:Main2}, we have $\nu(\mu)=\lambda(\mu)$ for any $\mu>0$ and, as a consequence, $\mu_2=\mu_3$.\end{Lemm}
%-----------------------------------------------------------------------------
\begin{proof} Assume first that $p>1$. The proof is based on two ways of estimating the quantity
\[
\mathcal A=\nrm{\nabla u}2^2+\lambda\,\nrm u2^2-\intg{\phi\,|u|^2}\,.
\]
On the one hand we may use H\"older's inequality to estimate
\[
\intg{\phi\,|u|^2}\le\nrm\phi q\,\nrm u{p+1}^2
\]
with $q=\frac{p+1}{p-1}$ and get
\[
\mathcal A\ge\nrm{\nabla u}2^2+\lambda\,\nrm u2^2-\mu\,\nrm u{p+1}^2
\]
with $\mu=\nrm\phi q$. Using $u\equiv1$ as a test function, we observe that the lowest eigenvalue $\lambda_1(\Omega,-\phi)$ of the Schr\"odinger operator $-\Delta-\phi$ is nonpositive. With $u=u_1$ an eigenfunction associated with $\lambda_1(\Omega,-\phi)$, we know that
\[
\mathcal A=\(\lambda-|\lambda_1(\Omega,-\phi)|\)\nrm u2^2\ge\nrm{\nabla u}2^2+\lambda\,\nrm u2^2-\mu\,\nrm u{p+1}^2
\]
is nonnegative if $\lambda=\lambda(\mu)$, thus proving that $\lambda(\mu)-|\lambda_1(\Omega,-\phi)|\ge0$ and hence
\[
\lambda(\mu)\ge\nu(\mu)\,.
\]
On the other hand, with $\phi=\mu\,u^{p-1}/\nrm u{p+1}^{p-1}$, we observe that
\begin{multline*}
0=\mathcal A=\nrm{\nabla u}2^2+\lambda\,\nrm u2^2-\mu\,\nrm u{p+1}^2\\
\ge\(\lambda-|\lambda_1(\Omega,-\phi)|\)\,\nrm u2^2\ge\(\lambda-\nu(\mu)\)\,\nrm u2^2
\end{multline*}
if we take $\mu=\mu(\lambda)$ and $u$ the corresponding optimal function. This proves that
\[
\lambda(\mu)\leq \nu(\mu)\,,
\]
which concludes the proof when $p>1$.

A similar computation can be done if $p<1$, based on the H\"older inequality
\[
\intg{u^{p+1}}\le\intg{u^{p+1}\,\phi^\frac{p+1}2\,\phi^{-\frac{p+1}2}}\le\(\intg{|u|^2\,\phi}\)^\frac{p+1}2\,\nrm{\phi^{-1}}q^\frac{p+1}2
\]
with $q=\frac{1+p}{1-p}$, that is
\[
\intg{|u|^2\,\phi}\ge\mu\,\nrm u{p+1}^2
\]
with $\mu^{-1}=\nrm{\phi^{-1}}q$. With
\[
\mathcal A=\nrm{\nabla u}2^2-\lambda\,\nrm u2^2+\intg{\phi\,|u|^2}\,,
\]
the computation is parallel to the one of the case $p>1$. Also see~\cite{DolEsLa-APDE2014} for similar estimates. \end{proof}
%-----------------------------------------------------------------------------
\begin{Lemm}\label{Lem:4} Under the assumptions of Theorem~\ref{Thm:Main2}, we have $\mu_1\le\mu_2$.\end{Lemm}
%-----------------------------------------------------------------------------
\begin{proof} Let $u$ be an optimal function for~\eqref{P2}. It can be taken nonnegative without restriction and solves the Euler-Lagrange equation
\be{EL}
-\,\varepsilon(p)\,\Delta u+\lambda\,u-\mu\,\frac{u^p}{\nrm u{p+1}^{p-1}}=0\quad\mbox{in}\quad\Omega\,,\quad\partial_nu=0\quad\mbox{on}\quad\partial\Omega\,,
\ee
where $\lambda=\lambda(\mu)$ or equivalently $\mu=\mu(\lambda)$. By homogeneity, we can fix $\nrm u{p+1}$ as we wish and may choose $\nrm u{p+1}^{p-1}=\mu$, hence concluding that $u$ is constant if $\mu\le\mu_1$ and, as a consequence, $\lambda(\mu)=\mu$, thus proving that $\mu\le\mu_2$. The conclusion follows. \end{proof}

%%%%%%%%%%%%%%%%%%%%%%%%%%%%%%%%%%%%%%%%%%%%%%%%%%%%%%%%%%%%%%%%%%%%%%%%%%%%%%
%%%%%%%%%%%%%%%%%%%%%%%%%%%%%%%%%%%%%%%%%%%%%%%%%%%%%%%%%%%%%%%%%%%%%%%%%%%%%%
\section{Estimates based on the heat equation}\label{Sec:Linear}

We use the Bakry-Emery method to prove some results that are slightly weaker than the assertion of Theorem~\ref{Thm:Main2} but the method is of its own independent interest. Except for the precise value of the constant, the following result can be found in~\cite{pre05312043} (also see earlier references therein).
%-----------------------------------------------------------------------------
\begin{Lemm}\label{Lem:FirstBis} Let $d\ge1$. Assume that $\Omega$ is a bounded convex domain such that $|\Omega|=1$. For any $p\in(0,1)$, for any $u\in \H^1(\Omega)$ such that $\partial_nu=0$ on $\partial\Omega$, we have
\[\label{Ineq:InterpBE}
\nrm{\nabla u}2^2\ge\lambda_2\,\left[\nrm u{2}^2-\nrm u{p+1}^2\right]\,.
\]
\end{Lemm}
%-----------------------------------------------------------------------------
In this section and in the next section, we are going to use the \emph{carr\'e du champ} method of D.~Bakry and M.~Emery in two different ways. Our goal is to prove that the functional $\mathcal J_\Lambda$ defined by~\eqref{JLambda} is nonnegative for some specific value of $\Lambda>0$.
\\
$\bullet$ In the \emph{parabolic} perspective, we will consider a flow $t\mapsto u(t,\cdot)$ and prove that
\[
\frac d{dt}\mathcal J_\Lambda[u(t,\cdot)]\le-\,\beta^2\,\mathcal R[u(t,\cdot)]
\]
for some non-zero parameter $\beta$ and some nonnegative functional $\mathcal R$. Since the flow drives the solutions towards constant functions, for which $\mathcal J_\Lambda$ takes the value $0$, we henceforth deduce that
\[
\mathcal J_\Lambda[u(t,\cdot)]\ge\lim_{s\to+\infty}\mathcal J_\Lambda[u(s,\cdot)]=0\quad\forall\,t\ge0\,.
\]
As a consequence, $\mathcal J_\Lambda[u_0]\ge0$ holds true for any initial datum $u(t=0,\cdot)=u_0\in\H^1(\Omega)$, which establishes the inequality. This approach has the advantage to provide for free a remainder term, since we know that
\[
\mathcal J_\Lambda[u_0]\ge\beta^2\int_0^{+\infty}\mathcal R[u(t,\cdot)]\,dt\,.
\]
The main disadvantage of the parabolic point of view is that it relies on the existence of a global and smooth enough solution. At least, this is compensated by the fact that one can take the initial datum as smooth as desired, prove the inequality and argue by density in $\H^1(\Omega)$. Such issues are somewhat standard and have been commented, for instance, in~\cite{MR2459454}.
\\
$\bullet$ Alternatively, one can adopt an \emph{elliptic} perspective. Since, as we shall see, \hbox{$\mathcal R[u]=0$} holds if and only if $u$ is constant on $\Omega$, it is enough to consider an optimal function for $\mathcal J_\Lambda$, which is known to exist by standard compactness methods for any exponent $p$ in the subcritical range, or even a positive critical point. The case of the critical exponent is more subtle, but can also be dealt with using techniques of the calculus of variations. An extremal function $u$ solves an Euler-Lagrange equation, which can be tested by a perturbation corresponding to the direction given by the flow. From a formal viewpoint, this amounts to take the solution to the flow problem with initial datum $u_0$ and to compute $\frac d{dt}\mathcal J_\Lambda[u(t,\cdot)]$ at \hbox{$t=0$}. However, no existence theory for the evolution equation is required and one can rely on the additional regularity properties that the function $u\in\H^1(\Omega)$ inherits as a solution to the Euler-Lagrange equation. The elliptic regularity theory \emph{\`a la} de Giorgi-Nash-Moser is also somewhat standard but requires some care. The interested reader is invited to refer, for instance, to~\cite{DEL2015} for further details on the application of this strategy, or to~\cite{1603} for a more heuristic introduction to the method.

In practice, we will use the two pictures without further notice. Detailed justifications and adaptations are left to the reader. Algebraically, in terms of integration by parts or tensor manipulations, the two methods are equivalent, and we shall focus on the these computations, which explain why the method works but also underlines its limitations.

\begin{proof}[Proof of Lemma~\ref{Lem:FirstBis}] We give a proof based on the entropy -- entropy production method. It is enough to prove the result for nonnegative functions $u$ since the inequality for $|u|$ implies the inequality for $u$. By density, we may assume that $u$ is smooth. According to~\cite{pre05312043}, if $v$ is a nonnegative solution of the heat equation
\[
\frac{\partial v}{\partial t}=\Delta v
\]
on $\Omega$ with homogeneous Neumann boundary conditions, then $v=u^{p+1}$ is such that
\[
\frac d{dt}\intg{\frac{v^r-M^r}{r-1}}=-\frac4r\intg{|\nabla u|^2}
\]
with $M:=\intg v$ and $r=2/(p+1)$. With this change of variables, $u$ solves
\[
\frac{\partial u}{\partial t}=\Delta u+\frac{2-r}r\,\frac{|\nabla u|^2}u
\]
and we find that
\begin{multline}\label{Id:BE}
-\,\frac12\,\frac d{dt}\intg{|\nabla u|^2}\\
=\intg{|\Delta u|^2}\,+\,p\intg{\frac{|\nabla u|^4}{u^2}}-2\,p\intg{\mathrm{Hess}\,u\!:\!\frac{\nabla u\otimes\nabla u}u}\,,
\end{multline}
that is,
\begin{multline*}
-\,\frac12\,\frac d{dt}\intg{|\nabla u|^2}=2\,\frac{r-1}r\intg{|\Delta u|^2}+p\intg{\(|\Delta u|^2-|\mathrm{Hess}\,u|^2\)}\\
+p\intg{\Big|\mathrm{Hess}\,u-\frac1u\,\nabla u\otimes\nabla u\Big|^2}\,,
\end{multline*}
and finally, using Lemma~\ref{Lem:Convex},
\[
\frac d{dt}\intg{|\nabla u|^2}\le-\,4\,\frac{r-1}r\intg{|\Delta u|^2}\le-\,4\,\frac{r-1}r\,\lambda_2\intg{|\nabla u|^2}
\]
where the last inequality follows from Lemma~\ref{Lem:LinInterp}, Ineq.~\eqref{Ineq:First}, thus proving the result for any $p=(2-r)/r\in(0,1)$. Indeed, with previous notations, we have shown that $\intg{|\nabla u|^2}$ is exponentially decaying. Hence
\[
\intg{\frac{v^r-M^r}{r-1}}=\frac1{r-1}\(\nrm u2^2-\nrm u{p+1}^2\)
\]
also converges to $0$ as $t\to\infty$ and
\[
\frac d{dt}\left[\nrm{\nabla u}2^2-\mu\intg{\frac{v^r-M^r}{r-1}}\right]\le\(-\,4\,\frac{r-1}r\,\lambda_2+\frac4r\,\mu\)\intg{|\nabla u|^2}
\]
is nonpositive if $\mu\le(r-1)\,\lambda_2$. Altogether, we have shown that
\[
\nrm{\nabla u}2^2-\lambda_2\(\nrm u2^2-\nrm u{p+1}^2\)
\]
is nonincreasing with limit $0$, which concludes the proof.\end{proof}

If $d\ge2$, better result can be obtained by considering the traceless quantities as in~\cite{Dolbeault20141338}. Let us introduce
\begin{eqnarray}\label{M-L}
&&\mathrm M[u]:=\frac{\nabla u\otimes\nabla u}u-\frac1d\,\frac{|\nabla u|^2}u\,\mathrm{Id}\,,\label{M}\\
&&\mathrm L\,u:=\mathrm{Hess}\,u-\frac1d\,\Delta u\,\mathrm{Id}\,,\label{L}
\end{eqnarray}
and define $p^\sharp:=\frac{d\,(d+2)}{(d-1)^2}$ so that
\[
\vartheta(p,d):=p\,\frac{(d-1)^2}{d\,(d+2)}
\]
satisfies $\vartheta(p,d)<1$ for any $p\in(0,p^\sharp)$. Notice that $p^\sharp+1=\frac{2\,d^2+1}{(d-1)^2}$ is the threshold value that has been found in~\cite{MR808640} (also see~\cite{DEKL2012,DEKL}).
%-----------------------------------------------------------------------------
\begin{Lemm}\label{Lem:Second}Let $d\ge2$. Assume that $\Omega$ is a bounded convex domain such that $|\Omega|=1$. For any $p\in(0,p^\sharp)$, for any $u\in\H^1(\Omega)$ such that $\partial_nu=0$ on $\partial\Omega$, we have
\[
\nrm{\nabla u}2^2\ge\frac12\,\big(1-\vartheta(p,d)\big)\,\lambda_2\,\frac{\nrm u{p+1}^2-\nrm u2^2}{p-1}
\]
if $p\neq1$ and, in the limit case $p=1$,
\[
\nrm{\nabla u}2^2\ge\frac14\,\big(1-\vartheta(1,d)\big)\,\lambda_2\intg{|u|^2\,\log\(\frac{|u|^2}{\nrm u2^2}\)}\,.
\]
\end{Lemm}
%-----------------------------------------------------------------------------
The range of $p$ covered in Lemma~\ref{Lem:Second} is larger than the range covered in Lemma~\ref{Lem:FirstBis}, but the constant is also better if $p\in\big(d\,(d+2)/(d^2+6\,p-1),1\big)$ because, in that case, $(1-\vartheta(p,d))/(p-1)>2$.

\begin{proof} We use the same conventions as in the proof of Lemma~\ref{Lem:FirstBis}. Let us first observe that
\begin{eqnarray*}
&&|\mathrm M[u]|^2=\(1-\frac1d\)\frac{|\nabla u|^4}{u^2}\,,\\
&&|\mathrm L\,u|^2=|\mathrm{Hess}\,u|^2-\frac1d\,(\Delta u)^2\,.\\
\end{eqnarray*}
Since $\mathrm{Hess}\,u=\mathrm L\,u+\frac1d\,\Delta u\,\mathrm{Id}$, we have that
\[
\mathrm{Hess}\,u\!:\!\frac{\nabla u\otimes\nabla u}u=\mathrm L\,u\!:\!\frac{\nabla u\otimes\nabla u}u+\frac1d\,\Delta u\,\frac{|\nabla u|^2}u=\mathrm L\,u\!:\!\mathrm M[u]+\frac1d\,\Delta u\,\frac{|\nabla u|^2}u
\]
because $\mathrm L\,u$ is traceless. An integration by parts shows that
\begin{multline*}
\intg{\Delta u\,\frac{|\nabla u|^2}u}=\intg{\frac{|\nabla u|^4}{u^2}}-\,2\intg{\mathrm{Hess}\,u\!:\!\frac{\nabla u\otimes\nabla u}u}\\
=\frac d{d-1}\intg{|\mathrm M[u]|^2}-\,2\intg{\mathrm{Hess}\,u\!:\!\frac{\nabla u\otimes\nabla u}u}
\end{multline*}
so that we get
\[
\frac{d+2}d\intg{\mathrm{Hess}\,u\!:\!\frac{\nabla u\otimes\nabla u}u}=\intg{\mathrm L\,u\!:\!\mathrm M[u]}+\frac1{d-1}\intg{|\mathrm M[u]|^2}\,,
\]
hence
\begin{multline*}
\intg{\mathrm{Hess}\,u\!:\!\frac{\nabla u\otimes\nabla u}u}=\frac d{d+2}\intg{\mathrm L\,u\!:\!\mathrm M[u]}\\
+\frac d{(d-1)\,(d+2)}\intg{|\mathrm M[u]|^2}\,.
\end{multline*}
Now let us come back to the proof of Lemma~\ref{Lem:FirstBis}. From~\eqref{Id:BE}, we read that
\begin{eqnarray*}
&&\kern-40pt-\,\frac12\,\frac d{dt}\intg{|\nabla u|^2}\\&\kern-5pt=&\kern-5pt\intg{|\Delta u|^2}+p\intg{\frac{|\nabla u|^4}{u^2}}-2\,p\intg{\mathrm{Hess}\,u\!:\!\frac{\nabla u\otimes\nabla u}u}\\
&\kern-5pt=&\kern-5pt\intg{|\Delta u|^2}+\frac{p\,d}{d-1}\intg{|\mathrm M[u]|^2}\\
&\kern-5pt&\kern-5pt-\,2\,p\(\frac d{d+2}\intg{\mathrm L\,u\!:\!\mathrm M[u]}+\frac d{(d-1)\,(d+2)}\intg{|\mathrm M[u]|^2}\)\\
&\kern-5pt=&\kern-5pt\intg{|\Delta u|^2}-p\,\frac{d-1}{d+2}\intg{|\mathrm L\,u|^2}\\
&\kern-5pt&\kern-5pt+\,\frac{p\,d^2}{(d-1)\,(d+2)}\intg{\left|\mathrm M[u]-\frac{d-1}d\,\mathrm L\,u\right|^2}\,.
\end{eqnarray*}
We know from Lemma~\ref{Lem:Convex} that
\[
\intg{(\Delta u)^2}\ge\intg{|\mathrm{Hess}\,u|^2}=\intg{|\mathrm L\,u|^2}+\frac1d\intg{(\Delta u)^2}\,,
\]
\emph{i.e.},
\[
\intg{(\Delta u)^2}\ge\frac d{d-1}\intg{|\mathrm L\,u|^2}\,.
\]
Altogether, this proves that, for any $\theta\in(0,1)$,
\[
-\,\frac12\,\frac d{dt}\intg{|\nabla u|^2}\ge(1-\theta)\intg{|\Delta u|^2}+\(\frac{\theta\,d}{d-1}-p\,\frac{d-1}{d+2}\)\intg{|\mathrm L\,u|^2}
\]
and finally, with $\theta=\vartheta(p,d)$ and using~\eqref{Ineq:First},
\[
\frac d{dt}\left[\nrm{\nabla u}2^2-\mu\intg{\frac{v^r-M^r}{r-1}}\right]\le\(-\,(1-\theta)\,\lambda_2+\frac4r\,\mu\)\intg{|\nabla u|^2}
\]
is nonpositive if
\[
\mu\le\frac r4\,\big(1-\vartheta(p,d)\big)\,\lambda_2=\frac{1-\vartheta(p,d)}{2\,(p+1)}\,\lambda_2\,.
\]
Since $r-1=(1-p)/(1+p)$, this concludes the proof if $p\neq1$. The case $p=1$ is obtained by passing to the limit as $p\to1$.\end{proof}

%%%%%%%%%%%%%%%%%%%%%%%%%%%%%%%%%%%%%%%%%%%%%%%%%%%%%%%%%%%%%%%%%%%%%%%%%%%%%%
%%%%%%%%%%%%%%%%%%%%%%%%%%%%%%%%%%%%%%%%%%%%%%%%%%%%%%%%%%%%%%%%%%%%%%%%%%%%%%
\section{Estimates based on nonlinear diffusion equations}\label{Sec:Nonlinear}

%-----------------------------------------------------------------------------
\begin{Lemm}\label{Lem:Lower} Assume that $d\ge2$ and $\Omega$ is a bounded convex domain in $\R^d$ with smooth boundary such that $|\Omega|=1$. Then we have
\[
\frac{1-\theta_\star(p,d)}{|p-1|}\,\lambda_2\le\mu_1\,.
\]
\end{Lemm}
%-----------------------------------------------------------------------------
\begin{proof} This bound is inspired from~\cite{MR1412446,MR1338283,MR1631581,Dolbeault20141338}. Let us give the main steps of the proof. Here we do it at the level of the nonlinear elliptic PDE. Flows will be introduced afterwards, with the intent of providing improvements.

Let us consider the solution $u$ to~\eqref{Eq3} and define a function $v$ such that $v^\beta=u$ for some exponent $\beta$ to be chosen later. Then $v$ solves
\be{Eq3bis}
-\,\varepsilon(p)\(\Delta v+(\beta-1)\,\frac{|\nabla v|^2}v\)+\lambda\,v-v^\kappa=0\quad\mbox{in}\quad\Omega
\ee
with homogeneous Neumann boundary conditions
\be{NBC}
\partial_nv=0\quad\mbox{on}\quad\partial\Omega\,.
\ee
Here
\be{kappa}
\kappa=\beta\,(p-1)+1\,.
\ee
If we multiply the equation by $\(\Delta v+\kappa\,|\nabla v|^2/v\)$ and integrate by parts, then the nonlinear term disappears and we are left with the identity
\begin{multline*}
\intg{(\Delta v)^2}+(\kappa+\beta-1)\intg{\Delta v\,\frac{|\nabla v|^2}v}+\kappa\,(\beta-1)\intg{\frac{|\nabla v|^4}{v^2}}\\
-\lambda\,|p-1|\intg{|\nabla v|^2}=0\,.
\end{multline*}
Using~\eqref{M}-\eqref{L}, let us define
\[
\mathrm Q[v]:=\mathrm L\,v-\frac{(d-1)\,(p-1)}{\theta\,(d+3-p)}\,\mathrm M[v]\,.
\]
The case of a compact manifold has been dealt with in~\cite{Dolbeault20141338}. The main difference is that there is no Ricci curvature in case of a domain in $\R^d$, but one has to take into account the boundary terms. As in the proof of Lemma~\ref{Lem:Second}, the main idea is to rearrange the various terms as a sum of squares of traceless quantities. The computations for $v$ are very similar to those done in Section~\ref{Sec:Linear}, so we shall skip the details. The reader is invited to check that
\begin{multline*}
\theta\(\intg{(\Delta v)^2}-\intg{|\mathrm{Hess}\,v|^2}\)+\frac{\theta\,d}{d-1}\intg{|\mathrm Q[v]|^2}\\
+(1-\theta)\intg{(\Delta v)^2}-\lambda\,|p-1|\intg{|\nabla v|^2}=0
\end{multline*}
if $\theta=\theta_\star(p,d)=\frac{(d-1)^2\,p}{d\,(d+2)+p}$ and $\beta=\frac{d+2}{d+2-p}$. In the previous identity, the first term is nonnegative by Lemma~\ref{Lem:Convex}, the second term is the integral of a square and is therefore nonnegative, and the sum of the last ones is positive according to Lemma~\ref{Lem:LinInterp} if $(1-\theta)\intg{(\Delta v)^2}-\lambda\,|p-1|\int_\Omega|\nabla v|^2>0$, unless $\nabla v=0$ a.e. Notice that the convexity of $\Omega$ is required to apply Lemma~\ref{Lem:Convex}.

We may notice that $p=d+2$ has to be excluded in order to define $\beta$, and this may occur if $d=2$. However, by working directly on $u$, it is possible to cover this case as well. This is indeed purely technical, because of the change of variables $u=v^\beta$. Alternatively, it is enough to observe that the inequality holds for any $p\neq d+2$ and argue by continuity with respect to~$p$.\end{proof}

\begin{proof}[Proof of Theorem~\ref{Thm:Main2}] Since the exponent $p$ is in the sub-critical range, it is classical that the functional $\mathcal J_\Lambda$ has a minimizer $u$. Up to a normalization $v=u^{1/\beta}$ solves~\eqref{Eq3bis}. If $\Lambda=\lambda\,|p-1|<\Lambda_\star$, then $u$ is constant by Lemma~\ref{Lem:Lower}, and we are therefore in the case $\lambda=\mu(\lambda)$ of Theorem~\ref{Thm:Main1} if $\lambda\le\frac{1-\theta_\star(p,d)}{|p-1|}\,\lambda_2$. Combined with the results of Theorem~\ref{Thm:Main1} and Lemma~\ref{Lem:Upper}, this completes the proof of Theorem~\ref{Thm:Main2}.

\end{proof}

\medskip For later purpose (see Section~\ref{Sec:Improved}), let us consider the proof based on the flow. With $\lambda=\Lambda/|p-1|$, we may consider the functional $u\mapsto\mathcal J_\Lambda[u]$ defined by~\eqref{JLambda} with $u=v^\beta$ and evolve it according to
\be{flow}
\frac{\partial v}{\partial t}=v^{2-2\,\beta} \(\Delta v+\kappa\,\frac{|\nabla v|^2}v\)\,.
\ee
We also assume that~\eqref{NBC} hold for any $t\ge0$. This flow has the nice property that
\[
\frac d{dt}\intg{u^{p+1}}=\frac d{dt}\intg{v^{\beta\,(p+1)}}=0
\]
if $\kappa$ is given by~\eqref{kappa}, and a simple computation shows that
\begin{multline*}
-\frac1{\beta^2}\,\frac d{dt}\mathcal J_\Lambda[v^\beta]=\intg{(\Delta v)^2}+(\kappa+\beta-1)\intg{\Delta v\,\frac{|\nabla v|^2}v}\\
+\kappa\,(\beta-1)\intg{\frac{|\nabla v|^4}{v^2}}-\lambda\,|p-1|\intg{|\nabla v|^2}=0\,.
\end{multline*}
The same choices of $\beta$ and $\theta$ as in the proof of Lemma~\ref{Lem:Lower} allow us to conclude, but it is interesting to discuss the possible values of $\beta$ and $\theta$ which guarantee that $\frac d{dt}\mathcal J_\Lambda[v^\beta]\le0$ unless $v$ is a constant. As in~\cite{Dolbeault20141338}, elementary computations show that
\begin{multline*}
-\frac1{\beta^2}\,\frac d{dt}\mathcal J_\Lambda[v^\beta]\\
=\theta\(\intg{(\Delta v)^2}-\intg{|\mathrm{Hess}\,v|^2}\)+\frac{\theta\,d}{d-1}\intg{|\mathrm Q[v]|^2}\\
+R\intg{\frac{|\nabla v|^4}{v^2}}+(1-\theta)\intg{(\Delta v)^2}-\lambda\,|p-1|\intg{|\nabla v|^2}
\end{multline*}
where $\mathrm Q[u]$ is now defined by 
\[
\mathrm Q[u]:=\mathrm L\,u-\frac1\theta\,\frac{d-1}{d+2}\,(\kappa+\beta-1)\left[\frac{\nabla u\otimes\nabla u}u-\frac1d\,\frac{|\nabla u|^2}u\,\mathrm{Id}\right]
\]
and
\[
R:=-\frac1\theta\(\frac{d-1}{d+2}\)^2(\kappa+\beta-1)^2+\kappa\,(\beta-1)+(\kappa+\beta-1)\,\frac d{d+2}\,.
\]
After replacing $\kappa$ by its value according to~\eqref{kappa}, we obtain that the equation
\[
0=R=\left[\(\frac{d-1}{d+2}\)^2\frac{p^2}\theta-p+1\right]\,\beta^2-2\(1-\frac p{d+1}\)\beta+1
\]
has two roots $\beta_\pm(\theta,p,d)$ if $\theta\in\big(\theta_\star(p,d),1\big)$ and $R>0$ if $\beta\in(\beta_-,\beta_+)$. As in the linear case (proof of Lemma~\ref{Lem:Second}), we also know from Lemma~\ref{Lem:Convex} that
\[
\intg{(\Delta v)^2}-\intg{|\mathrm{Hess}\,v|^2}\ge0
\]
and this is precisely where we take into account boundary terms and use the assumption that $\Omega$ is convex. Summarizing, we arrive at the following result.
%-----------------------------------------------------------------------------
\begin{Prop}\label{Prop:mu} With the above notations, if $\Omega$ is a bounded convex domain such that $|\Omega|=1$, for any $\theta\in\big(\theta_\star(p,d),1\big)$ and any $\beta\in\big(\beta_-(\theta,p,d),\beta_+(\theta,p,d)\big)$, we have
\[
\frac d{dt}\mathcal J_\Lambda[v^\beta]\le-\,R\,\beta^2\intg{\frac{|\nabla v|^4}{v^2}}
\]
if $v$ is a solution to~\eqref{flow}.\end{Prop}
%-----------------------------------------------------------------------------
When $\theta=\theta_\star(p,d)$, the reader is invited to check that $\beta_-=\beta_+=\beta$. The computations in the proof of Lemma~\ref{Lem:Lower} can now be reinterpreted in the framework of the flow defined by~\eqref{flow}. Up to the change of unknown function $u=v^\beta$, any solution to~\eqref{Eq3} is stationary with respect to~\eqref{flow} and then all  our computations amount to write that $\frac d{dt}\mathcal J_\Lambda[v^\beta]=0$ is possible only if $v$ is a constant.

%%%%%%%%%%%%%%%%%%%%%%%%%%%%%%%%%%%%%%%%%%%%%%%%%%%%%%%%%%%%%%%%%%%%%%%%%%%%%%
%%%%%%%%%%%%%%%%%%%%%%%%%%%%%%%%%%%%%%%%%%%%%%%%%%%%%%%%%%%%%%%%%%%%%%%%%%%%%%
\section{Further considerations}\label{Sec:further}

Let us conclude this paper by a series of remarks. Section~\ref{Sec:Threshold} is devoted to the question of the non-optimality in the lower bound of Theorem~\ref{Thm:Main2}. Spectral methods are introduced in Section~\ref{Sec:Nelson} and provide us with an alternative method to establish~\eqref{GenInterpIneq} with $p\in(0,1)$ when the constants in the extremal cases $p=0$ (Poincar\'e inequality) and $p=1$ (logarithmic Sobolev inequality) are known. The last estimates of Section~\ref{Sec:Improved} are based  on refinements of the nonlinear flow method and extend the case of the manifolds with positive curvature studied in~\cite{DEKL} to the setting of a bounded convex domain with homogeneous Neumann boundary conditions.

%%%%%%%%%%%%%%%%%%%%%%%%%%%%%%%%%%%%%%%%%%%%%%%%%%%%%%%%%%%%%%%%%%%%%%%%%%%%%%
\subsection{The threshold case}\label{Sec:Threshold}

The following result complements those of Theorem~\ref{Thm:Main2}.
%-----------------------------------------------------------------------------
\begin{Prop}\label{Prop:Threshold} With $\theta_\star(p,d)$ defined by~\eqref{thetastar} and $\Lambda_\star$ given as the best constant in~\eqref{GenInterpIneq}, if $\Omega$ is a bounded convex domain such that $|\Omega|=1$, we have that
\[
\big[1-\theta_\star(p,d)\big]\,\lambda_2<\Lambda_\star\le\lambda_2\,.
\]
\end{Prop}
%-----------------------------------------------------------------------------
\begin{proof} The proof goes along the same lines as~\cite[Theorem~4]{Dolbeault20141338}. We argue by contradiction and assume first that
\[
\big[1-\theta_\star(p,d)\big]\,\lambda_2=\Lambda_*
\]
and that there is a nontrivial solution to~\eqref{Eq3} for $\lambda\,|p-1|=\Lambda=\Lambda_\star$. Then $\mathcal J_\Lambda[v^\beta]$ is constant with respect to $t$ if $\lambda=\Lambda_\star\,|p-1|$ and $v$ is a solution of~\eqref{flow} with initial datum $v_0$ such that $u=v_0^\beta$ is optimal for the functional inequality~\eqref{GenInterpIneq}. Since the limit of $v(t,\cdot)$ as $t\to\infty$ is a positive constant that can be approximated by the average of $v(t,\cdot)$, then
\[
0=\mathcal J_\Lambda[v(t,\cdot)]\sim\nrm{\nabla w}2^2-\lambda\,\nrm w2^2\ge\theta_\star(p,d)\,\lambda_2\,\nrm w2^2
\]
with $w=v-\intg{v(t,\cdot)}\neq0$, a contradiction. 

Alternatively we can use the elliptic point of view and consider non-trivial optimal functions $u_\lambda$ with $\lambda>\Lambda_\star$. As $\lambda\to\Lambda_\star$, $u_\lambda$ has to converge to a constant and we again reach a contradiction.\end{proof}

%%%%%%%%%%%%%%%%%%%%%%%%%%%%%%%%%%%%%%%%%%%%%%%%%%%%%%%%%%%%%%%%%%%%%%%%%%%%%%
\subsection{An interpolation between Poincar\'e and logarithmic Sobolev inequalities}\label{Sec:Nelson}

It is well known that inequalities~\eqref{GenInterpIneq} with $p\in(0,1)$ can be seen as a family of inequalities which interpolate between Poincar\'e and logarithmic Sobolev inequalities. See for instance~\cite{MR1796718}. Next, using the method introduced by W.~Beckner in~\cite{MR954373} in case of a Gaussian measure and later used for instance in \cite{ABD05,DEKL2012}, we are in position to get an estimate of the best constant in~\eqref{GenInterpIneq} for any $p\in(0,1)$ in terms of the best constant at the endpoints $p=0$ and $p=1$. To emphasize the dependence in the optimal constant in $p$, we shall denote it by $\Lambda_\star(p)$ and consistantly use $\Lambda_\star(1)$ as the optimal constant in the logarithmic Sobolev inequality~\eqref{LogSob}. We recall that the optimal constants in~\eqref{GenInterpIneq} are such that
\[
\Lambda_\star(p)\le\Lambda_\star(0)=\lambda_2
\]
for any $p\in(0,2^*-1)$, including in the case $p=1$ of the logarithmic Sobolev inequality. This can be checked easily as in the proof of Lemma~\ref{Lem:Upper} by using $u=1+\epsilon\,u_2$ as a test function, where $u_2$ is an eigenfunction associated with~$\lambda_2$, and by taking the limit as $\epsilon\to0$.
%-----------------------------------------------------------------------------
\begin{Prop}\label{Prop:4} Assume that $p\in(0,1)$ and $d\ge 1$. Then we have the estimate
\[
\Lambda_\star(p)\ge\frac{1-p}{1-p^\alpha}\,\lambda_2\quad\mbox{with}\quad\alpha=\frac{\lambda_2}{\Lambda_\star(1)}\,.
\]
 \end{Prop}
%-----------------------------------------------------------------------------
\begin{proof} Let us briefly sketch the proof which is based on two main steps. 

\medskip\noindent\emph{$1^{\rm st}$ step: Nelson's hypercontractivity result.\/} Based on the strategy of L.~Gross in~\cite{Gross75}, we first establish an optimal hypercontractivity result using~\eqref{LogSob}. On $\Omega$, let us consider the heat equation
\[
\frac{\partial f}{\partial t}=\Delta f
\]
with initial datum $f(t=0,\cdot)=u$, Neumann homogeneous boundary conditions and let $F(t):=\nrm {f(t,\cdot)}{Q(t)}$. The key computation goes as follows.
\begin{multline*}
\frac{F'}F=\frac d{dt}\,\log F(t)=\frac d{dt}\,\left[\frac 1{Q(t)}\,\log\(\intg{|f(t,\cdot)|^{Q(t)}}\)\right]\\
=\frac{Q'}{Q^2\,F^Q}\left[\intg{v^2\log\(\frac{v^2}{\intg{v^2}}\)}+4\,\frac{Q-1}{Q'}\,\intg{|\nabla v|^2}\right]
\end{multline*}
with $v:=|f|^{Q(t)/2}$. Assuming that $4\,\frac{Q-1}{Q'}=\frac 2{\Lambda_\star(1)}$, we find that
\[
\log\(\frac{Q(t)-1}p\)=2\,\Lambda_\star(1)\,t
\]
if we require that $Q(0)=p+1$. Let $t_*>0$ be such that $Q(t_*)=2$. As a consequence of the above computation, we observe that $F$ is non increasing by the logarithmic Sobolev inequality~\eqref{LogSob} and get that
\be{Ineq:Nelson}
\nrm{f(t_*,\cdot)}2\le\nrm up\quad\mbox{if}\quad\frac 1p=e^{2\,\Lambda_\star(1)\,t_*}\;.
\ee

\medskip\noindent\emph{$2^{\rm nd}$ step: Spectral decomposition.\/} Let $u=\sum_{k\ge1}u_k$ be a decomposition of the initial datum on the eigenspaces of $-\Delta$ with Neumann boundary conditions and denote by $\lambda_k$ the ordered sequence of the eigenvalues: $-\Delta u_k=\lambda_k\,u_k$. Let $a_k=\nrm{u_k}2^2$. As a straightforward consequence of this decomposition, we know that $\nrm u2^2=\sum_{k\ge1}a_k$, $\nrm{\nabla u}2^2=\sum_{k\ge1}\lambda_k\,a_k$,
\[
\nrm{f(t_*,\cdot)}2^2=\sum_{k\ge1}a_k\,e^{-2\,\lambda_k\,t_*}\;.
\]
Using~\eqref{Ineq:Nelson}, it follows that
\[
\frac{\nrm u2^2-\nrm up^2}{1-p}\le\frac{\nrm u2^2-\nrm{f(t_*,\cdot)}2}{1-p}
\]
where the right hand side can be rewritten as
\begin{multline*}
\frac 1{1-p}\sum_{k\ge2}\lambda_k\,a_k\,\frac{1-e^{-2\,\lambda_k\,t_*}}{\lambda_k}\le\frac{1-e^{-2\,\lambda_2\,t_*}}{(1-p)\,\lambda_2}\sum_{k\ge2}\lambda_k\,a_k\\
=\frac{1-e^{-2\,\lambda_2\,t_*}}{(1-p)\,\lambda_2}\,\nrm{\nabla u}2^2\,.
\end{multline*}\end{proof}

Notice that the estimate of Proposition~\ref{Prop:4} allows us to recover the optimal values $\lambda_2$ and $\Lambda_\star(1)$ when passing to the limit in $\frac{1-p}{1-p^\alpha}\,\lambda_2$ as $p\to0$ and $p\to1$ respectively. Hence any improvement on the estimate of $\Lambda_\star(1)$ automatically produces an improvement upon the lower estimate in Theorem~\ref{Thm:Main2} at least in a neighborhood of $p=1_-$.

%%%%%%%%%%%%%%%%%%%%%%%%%%%%%%%%%%%%%%%%%%%%%%%%%%%%%%%%%%%%%%%%%%%%%%%%%%%%%%
\subsection{Improvements based on the nonlinear flow}\label{Sec:Improved}
Let us define the exponent
\[
\delta:=\frac{p+1+\,\beta\,(p-3)}{2\,\beta\,(p-1)}\,.
\]
Improvements of (\ref{GenInterpIneq}) can be obtained as in \cite{Demange-PhD,DEKL}, using the following interpolation lemma.
%-----------------------------------------------------------------------------
\begin{Lemm}\label{Lem:Demange} Assume that $\beta>1$, and $\beta\le\frac2{3-p}$ if $p<3$. For any $u=v^\beta\in\H^1(\Omega)$ such that $\nrm u{p+1}=1$, we have
\[
\intg{\frac{|\nabla v|^4}{v^2}}\ge\frac1{\beta^2}\,\frac{\intg{|\nabla u|^2}\,\intg{|\nabla v|^2}}{\(\intg{u^2}\)^\delta}\,.
\]
\end{Lemm}
%-----------------------------------------------------------------------------
\begin{proof} With $\frac12+\frac{\beta-1}{2\,\beta}+\frac1{2\,\beta}=1$, H\"older's inequality shows that
\begin{multline*}
\intg{|\nabla v|^2}=\intg{\frac{|\nabla v|^2}v\,1\,v}\\
\le\(\intg{\frac{|\nabla v|^4}{v^2}}\)^\frac12\(\intg 1\)^\frac{\beta-1}{2\,\beta}\(\intg{v^{2\,\beta}}\)^\frac1{2\,\beta}\,,
\end{multline*}
from which we deduce that
\be{Demange1}
\(\intg{\frac{|\nabla v|^4}{v^2}}\)^\frac12\ge\frac{\intg{|\nabla v|^2}}{\(\intg{u^2}\)^\frac1{2\,\beta}}
\ee
because $|\Omega|=1$. With $\frac12+\frac{\beta-1}{\beta\,(p-1)}+\frac{\beta\,(p-3)+2}{2\,\beta\,(p-1)}=1$, H\"older's inequality shows that
\begin{multline*}
\frac1{\beta^2}\intg{|\nabla(v^\beta)|^2}\\
=\intg{v^{2\,(\beta-1)}\,|\nabla v|^2}=\intg{\frac{|\nabla v|^2}v\,v^\frac{(p+1)\,(\beta-1)}{p-1}\cdot v^\frac{\beta\,(p-3)+2}{p-1}}\\
\le\(\intg{\frac{|\nabla v|^4}{v^2}}\)^\frac12\(\intg{v^{\beta\,(p+1)}}\)^\frac{\beta-1}{\beta\,(p-1)}\(\intg{v^{2\,\beta}}\)^\frac{\beta\,(p-3)+2}{2\,\beta\,(p-1)}\,,
\end{multline*}
from which we deduce that
\[
\(\intg{\frac{|\nabla v|^4}{v^2}}\)^\frac12\ge\frac1{\beta^2}\,\frac{\intg{|\nabla u|^2}}{\(\intg{u^2}\)^\frac{\beta\,(p-3)+2}{2\,\beta\,(p-1)}}\,.
\]
This inequality combined with~\eqref{Demange1} completes the proof.\end{proof}

For any $\beta>1$, we define
\[
\varphi(s):=\int_0^s\exp\left[\kappa\(\(1\,-\,(p-1)\,z\)^{1-\delta}-\(1\,-\,(p-1)\,s\)^{1-\delta}\)\right]\,dz
\]
where $\kappa=\tfrac{R}{\beta\,(\beta-1)\,(p+1)}$ and $R$ appears in Proposition~\ref{Prop:mu}, and let
\[
\Phi(s):=\big(1+(p-1)\,s\big)\,\varphi\!\(\frac s{1+(p-1)\,s}\)\,.
\]
%-----------------------------------------------------------------------------
\begin{Theo}\label{Thm:Improved} Assume that $\Omega$ is a bounded convex domain such that $|\Omega|=1$ and that one of the following conditions is satisfied:
\begin{enumerate}
\item[(i)] $d=2$ and $p\in(0,1)\cup(1,8+4\sqrt3)$,
\item[(ii)] $d\ge3$ and $p\in(0,1)\cup(1,2^*-1)$.
\end{enumerate}
With the notations of Proposition~\ref{Prop:mu}, for any $u\in\H^1(\Omega)$ be such that $\nrm u2=1$, we have the inequality
\[
\frac{1-\theta}{p-1}\,\lambda_2\;\Phi\!\(\frac{\nrm up^2-1}{p-1}\)\le\nrm{\nabla u}2^2
\]
for any $\theta\in\big(\theta_\star(p,d),1\big)$ and $\beta>1$ such that $\beta_-(\theta,p,d)<\beta<\beta_+(\theta,p,d)$.\end{Theo}
%-----------------------------------------------------------------------------
\begin{proof} Let us define $\Lambda=\frac{1-\theta}{p-1}\,\lambda_2$
\[
\mathsf e=\frac1{p-1}\,\left[\nrm u{p+1}^2-\nrm u2^2\right]\,,\quad\mathsf i:=\nrm{\nabla u}2^2
\]
so that
\[
\mathcal J_\Lambda[u]=\mathsf i-\,\Lambda\,\mathsf e\,.
\]
Using Proposition~\ref{Prop:mu} and Lemma~\ref{Lem:Demange}, we obtain the differential inequality
\[
\mathsf i'-\,\Lambda\,\mathsf e'-\frac{R}{2\,\beta^2}\,\frac{\mathsf i\,\mathsf e'}{\(1\,-\,(p-1)\,\mathsf e\)^\delta}\le0
\]
which can be rewritten as
\[
\frac d{dt}\big(\mathsf i\,\psi'(\mathsf e)-\Lambda\,\psi(\mathsf e)\big)\le0
\]
if $\varphi$ and $\psi$ are related by
\[
\varphi(\mathsf e):=\frac{\psi(\mathsf e)}{\psi'(\mathsf e)}\,.
\]
It is then elementary to check that $\varphi$ satisfies the ODE
\[
\varphi'=1-\varphi\,\frac{\psi''(\mathsf e)}{\psi'(\mathsf e)}=1+\varphi\,\frac{R}{2\,\beta^2}\,\(1\,-\,(p-1)\,\mathsf e\)^{-\delta}
\]
and $\varphi(0)=0$. 
\end{proof}

The reader interested in the precise ranges of the exponent $\beta$ and the values of $\theta$ is invited to refer to~\cite{DEKL} for more details.

%%%%%%%%%%%%%%%%%%%%%%%%%%%%%%%%%%%%%%%%%%%%%%%%%%%%%%%%%%%%%%%%%%%%%%%%%%%%%%
\subsection{Some concluding remarks and open questions}

Beyond the fact that we deal with a bounded domain with Neumann boundary conditions instead of a compact manifold with positive curvature, the lower estimate in Theorem~\ref{Thm:Main2} differs from the existing literature in several aspects. First of all we emphasize the fact that the convexity is needed for our method (Lemma~\ref{Lem:Convex}) but is certainly not necessary. The whole range of exponents corresponding to $2<p+1<2^*$ is covered as in~\cite{MR1412446,MR1134481,MR1631581,MR1338283} and the flow interpretation gives a nice framework, which is already present in the results of J.~Demange in~\cite{MR2381156} and has been emphasized in~\cite{Dolbeault20141338,DEKL}. Even better, the range $1<p+1<2$ is also covered, which is new in the context of bounded domains. As the problem is set on the Euclidean space, we have neither a curvature assumption nor pointwise CD($\rho$,$N$) conditions. What matters is the Poincar\'e constant, which was already taken into account in the papers of J.R.~Licois and L.~V\'eron in~\cite{MR1338283,MR1631581} and D.~Bakry and M.~Ledoux in~\cite{MR1412446} in the case of compact manifolds. However, we deal only with integral quantities and integrations by parts, as was emphasized in~\cite{Dolbeault20141338}, still in the compact manifolds case. Last but not least, the nonlinear flow approach is also based on the methods of~\cite{DEKL2012,Dolbeault20141338} for compact manifolds, but the results of Section~\ref{Sec:Improved} on the improved inequalities as the ones obtained in~\cite{DEKL} go beyond the results that have  been achieved so far by standard techniques of nonlinear elliptic equations.

\medskip By studying radial solutions to~\eqref{Eq3}, further results can be obtained using ODE techniques. For instance, if $p\in(0,1)$ and $\lambda>0$ is large enough, according to the \emph{compact support principle}, there are non-constant, radial solutions with compact support in a ball strictly contained in $\Omega$. This, in particular, provides us with an upper bound on $\Lambda_\star$. See~\cite{MR1938658,MR1387457,MR1629650,MR1715341}.

In dimension $d=1$, the computations are almost explicit. Scalings can be used so that the problem is equivalent to the case $\lambda=1$ on an interval with varying length. See for instance \cite{MR607786,MR1090827} for results in this direction.

\medskip Problems (P2) and (P3) are equivalent. An optimal function for (P2) solves~\eqref{EL}, and any solution of~\eqref{EL} is optimal as can be checked by multiplying the equation by $u$ and integrating on $\Omega$. The threshold for \emph{rigidity} in~\eqref{EL} is therefore $\lambda=\mu_2$. However, this problem is of different nature than the rigidity problem in (P1). Because all terms in~\eqref{EL} are $1$-homogenous, the normalization of $u$ in $\L^{p+1}(\Omega)$ is free and one can of course take $\nrm u{p+1}^{p-1}=\mu$ so that $u$ solves~\eqref{Eq3}. Rigidity in (P1) implies rigidity in (P2). The reverse implication is not true and, up to the multiplication by a constant, all solutions of~\eqref{EL} solve~\eqref{Eq3}, while the opposite is not true. In that sense, the set of solutions to~\eqref{EL} is larger, which explains why we only prove that $\mu_1\le\mu_2$.

\medskip As a conclusion, let us mention a few open questions. First of all a natural question would be to try to prove that $\mu_1=\mu_2$ in the statement of Theorem~\ref{Thm:Main2}: under which assumptions can this be done~? In the framework of compact manifolds, this is true in the case of the sphere, but it turns out that the lower estimate on $\mu_1$ given by the nonlinear flow method is then equal to $\Lambda_\star/|p-1|$, which is definitely a very peculiar case.

Are there cases for which $\mu_1<\mu_2$~? For more complex interpolation inequalities on a cylinder, it has been established in \cite{springerlink:10.1007/s00526-011-0394-y} that this happens and the interested reader is invited to refer to \cite{FreefemDolbeaultEsteban,0951-7715-27-3-435} for more details of qualitative nature. If $\Omega$ is a ball numerical computations when $d=2$ and $p=2$ also show that $\mu_1<\mu_2$ as long as the study is done within the radial setting, but the branch of solutions corresponding to $\mu(\lambda)<\lambda$ is generated by non-radial functions. If $\mu_1=\mu_2$, then $\mu_1$ is also a threshold value for the existence of non-constant solutions: for any $\mu>\mu_1$ such solutions indeed exist. Is this also what happens if $\mu_1<\mu_2$, or are there values of $\mu\in(\mu_1,\mu_2)$ such that all positive, or at least nonnegative, solutions are in fact constants~?
 
Branches of solutions and bifurcations have been the subject of numerous papers and we did not review the existing literature, but at least one can mention an interesting problem. We know that optimal potentials in (P3) are related with optimal functions in (P2). Is it possible to take advantage of the spectral information in Problem (P3) to get information on branches of solutions associated with (P1) ?

%%%%%%%%%%%%%%%%%%%%%%%%%%%%%%%%%%%%%%%%%%%%%%%%%%%%%%%%%%%%%%%%%%%%%%%%%%%%%%
%%%%%%%%%%%%%%%%%%%%%%%%%%%%%%%%%%%%%%%%%%%%%%%%%%%%%%%%%%%%%%%%%%%%%%%%%%%%%%
\par\medskip\centerline{\rule{2cm}{0.2mm}}\medskip\noindent
\copyright\,2016 by the authors. This paper may be reproduced, in its entirety, for non-commercial purposes.
%\par\medskip\centerline{\rule{2cm}{0.2mm}}\medskip
%%%%%%%%%%%%%%%%%%%%%%%%%%%%%%%%%%%%%%%%%%%%%%%%%%%%%%%%%%%%%%%%%%%%%%%%%%%%%%
%%%%%%%%%%%%%%%%%%%%%%%%%%%%%%%%%%%%%%%%%%%%%%%%%%%%%%%%%%%%%%%%%%%%%%%%%%%%%%

%%%%%%%%%%%%%%%%%%%%%%%%%%%%%%%%%%%%%%%%%%%%%%%%%%%%%%%%%%%%%%%%%%%%%%%%%%%%%%
\def\bysame{\leavevmode ---------\thinspace}
\makeatletter\if@francais\providecommand{\og}{<<~}\providecommand{\fg}{~>>}
\else\gdef\og{``}\gdef\fg{''}\fi\makeatother
\def\cdrandname{\&}
\providecommand\cdrnumero{no.~}
\providecommand{\cdredsname}{eds.}
\providecommand{\cdredname}{ed.}
\providecommand{\cdrchapname}{chap.}
\providecommand{\cdrmastersthesisname}{Memoir}
\providecommand{\cdrphdthesisname}{PhD Thesis}

\end{document}